\documentclass[12pt]{amsart}
\usepackage{amssymb}
\usepackage[color, notref, notcite]{showkeys}     % refs and labels
\definecolor{refkey}{gray}{.5}   % graylevel for refs
\definecolor{labeled}{gray}{.5} % graylevel for labels
\usepackage{graphicx}

\newtheorem{thm}{Theorem}[section]

\newtheorem{cor}[thm]{Corollary}
\newtheorem{prob}[thm]{Problem}
\newtheorem{lem}[thm]{Lemma}
\newtheorem{prop}[thm]{Proposition}

\theoremstyle{definition}
\newtheorem{defn}[thm]{Definition}

\newcommand{\ol}{\overline}
\newcommand{\sm}{\setminus}
\newcommand{\op}{orientation preserving}
\newcommand{\reals}{\mathbb{R}}
\newcommand{\R}{\mathbb{R}}

\newcommand{\A}{\mathcal{A}}
\newcommand{\p}{\mathcal{P}}

\newcommand{\hp}{\widehat{p}}
\newcommand{\zed}{\mathbb{Z}}
\newcommand{\C}{\mathbb{C}}
\newcommand{\sphere}{\mathbb{S}}
\newcommand{\ucirc}{\mathbb{S}^1}
\newcommand{\uc}{\mathbb{S}^1}
\newcommand{\var}[2]{\mathrm{Var(#1,#2)}}
\newcommand{\ind}[2]{\mathrm{Ind(#1,#2)}}

\newcommand{\win}{\mathrm{win}}
\newcommand{\0}{\emptyset}
\newcommand{\cl}{\mbox{Cl}}
\newcommand{\val}{\mbox{val}}
\newcommand{\Int}{\mbox{int}}
\newcommand{\bd}{\mbox{Bd}}
\newcommand{\dia}{\mbox{diam}}
\newcommand{\sh}{\mbox{Sh}}
\newcommand{\ch}{\mbox{Ch}}
\newcommand{\e}{\varepsilon}
\newcommand{\al}{\alpha}
\newcommand{\ba}{\beta}

\newcommand{\be}{\beta}

\newcommand{\si}{\sigma}

\newcommand{\da}{\delta}
\newcommand{\vp}{\varphi}
\newcommand{\la}{\lambda}
\newcommand{\nin}{\not\in}
\newcommand{\imp}{\mbox{Imp}}

\newcommand{\disk}{\mathbb{D}}
\newcommand{\real}{\mathbb{R}}
\newcommand{\RC}{\C^\infty}
\newcommand{\degree}{\text{degree}}

\newcommand{\iy}{\infty}
\newcommand{\g}{\mathbf{g}}
\begin{document}

\date{September 21, 2008}

\title[Fixed points]{Fixed points in non-invariant plane continua}
\author{Alexander~Blokh}

\thanks{The first author was partially
supported by NSF grant DMS-0456748}

\author{Lex~Oversteegen}

\thanks{The second author was partially  supported
by NSF grant DMS-0405774}

\address[Alexander~Blokh and Lex~Oversteegen]
{Department of Mathematics\\ University of Alabama at Birmingham\\
Birmingham, AL 35294-1170}

\email[Alexander~Blokh]{ablokh@math.uab.edu}
\email[Lex~Oversteegen]{overstee@math.uab.edu} 


\begin{abstract}
If $f:[a,b]\to \mathbb{R}$, with $a<b$, is continuous and such that $a$ and $b$
are mapped in opposite directions by $f$, then $f$ has a fixed point in $I$.
Suppose that $f:\mathbb{C}\to\mathbb{C}$ is map and $X$ is a continuum. We
extend the above for certain continuous maps of dendrites $X\to D, X\subset D$
and for positively oriented maps $f:X\to \mathbb{C}, X\subset \mathbb{C}$ with
the continuum $X$ not necessarily invariant. Then we show that in certain cases
a holomorphic map $f:\mathbb{C}\to\mathbb{C}$ must have a fixed point $a$ in a
continuum $X$ so that either $a\in \mathrm{Int}(X)$ or $f$ exhibits rotation at
$a$.
\end{abstract}
\maketitle

\section{Introduction}

By $\C$ we denote the plane and by $\RC$ the Riemann sphere. The fixed point
problem, attributed to \cite{ster35}, is one of the central problems in plane
topology. It can be formulated as follows.

\begin{prob}\label{fpt} Suppose that $f:\C\to \C$ is continuous and $f(X)\subset X$
for a non-separating continuum $X$. Does there always exist a fixed point in $X$?
\end{prob}

As is, Problem~\ref{fpt} is not yet solved. The most well-known particular case
for which it is solved is that of a map of a closed interval $I=[a, b]$, $a<b$
into itself in which case there must exist a fixed point in $I$. In fact, in
this case a more general result can be proven of which the existence of a fixed
point in an invariant interval is a consequence.

Namely, instead of considering a map $f:I\to I$ consider a map $f:I\to \R$ such
that either (a) $f(a)\ge a$ and $f(b)\le b$, or (b) $f(a)\le a$ and $f(b)\ge
b$. Then still there must exist a fixed point in $I$ which is an easy corollary
of the Intermediate Value Theorem applied to the function $f(x)-x$. Observe
that in this case $I$ need not be invariant under $f$. Observe also that
without the assumptions on the endpoints, the conclusion on the existence of a
fixed point inside $I$ cannot be made because, e.g., a shift map on $I$ does
not have fixed points at all. The conditions (a) and (b) above can be thought
of as boundary conditions imposing restrictions on where $f$ maps the boundary
points of $I$ in $\reals$.

Our main aim in this paper is to consider some other cases for which
Problem~\ref{fpt} is solved in affirmative (i.e., the existence of a fixed
point in an invariant continuum is established) and replace for them the
invariantness of the continuum by boundary conditions in the spirit of the
above ``interval version'' of Problem~\ref{fpt}. More precisely, instead of
assuming that $f(X)\subset X$ we will make some assumptions on the way that $f$
maps points of $\ol{f(X)\sm X}\cap X$.

First though let us discuss particular cases for which Problem~\ref{fpt} is
solved. They can be divided into two categories: either $X$ has specific
properties, or $f$ has specific properties. The most direct extension of the
``interval particular case'' of Problem~\ref{fpt} is, perhaps, the following
well known theorem (see for example\cite{nadl92}).

\begin{thm}\label{dendr}
If $f:D\to D$ is a continuous map of a dendrite into itself then it has a fixed
point.
\end{thm}

Here $f$ is just a continuous map but the continuum $D$ is very nice.
Theorem~\ref{dendr} can be generalized to the case when $f$ maps $D$ into a
dendrite $X\supset D$ and certain conditions onto the behavior of the points of
the boundary of $D$ in $X$ are fulfilled. This presents a ``non-invariant''
version of Problem~\ref{fpt} for dendrites and can be done in the spirit of the
interval case described above. Moreover, with some additional conditions it has
consequences related to the number of periodic points of $f$. The details and
exact statements can be found to Section 2 devoted to the dendrites.

Another direction is to consider specific maps of the plane on arbitrary
non-separating continua. Let us go over known results here. Cartwright and
Littlewood \cite{cartlitt51} have solved Problem~\ref{fpt} in affirmative for
orientation preserving homeomorphisms of the plane. This result was generalized
to all homeomorphisms by Bell \cite{bell78}. The existence of fixed points for
orientation preserving homeomorphisms of the \emph{entire plane} under various
conditions was considered in
\cite{brou12a,brow84a,fath87,fran92,guil94,fmot07}, and of a point of period
two for orientation reversing homeomorphisms in \cite{boni04}.

The result by Cartwright and Littlewood  deals with the case when $X$ is an
invariant continuum. In parallel with the interval case, we want to extend this
to a larger class of maps of the plane (i.e., not necessarily one-to-one) such
that certain ``boundary'' conditions are satisfied. Our tools are mainly based
on \cite{fmot07} and apply to \emph{positively oriented maps} which
generalize the notion of an orientation preserving homeomorphism (see the
precise definitions in Section~\ref{prel}). Our main topological results are
Theorems~\ref{fixpt} and ~\ref{locrot}. The precise conditions in them are
somewhat technical - after all, we need to describe ``boundary conditions'' of
maps of arbitrary non-separating continua. However it turns out that these
conditions are satisfied by holomorphic maps (in particular, polynomials),
allowing us to obtain a few corollaries in this case, essentially all dealing
with the existence of periodic points in certain parts of the Julia set of a
polynomial and degeneracy of certain impressions.

\section{Preliminaries}\label{prel}

A map $f:X\to Y$ is \emph{perfect} if for each compact set $K\subset Y$,
$f^{-1}(K)$ is also compact. \emph{All maps considered in this paper are
prefect.} Given a continuum $K\subset \C$, denote by $U_\iy(K)$ the component
of $\C\sm K$ containing infinity, and by $T(K)$ the \emph{topological hull of
$K$}, i.e. the set $\C\sm U_\iy(K)$. By $\uc$ we denote the unit circle which
we identify with $\real\mod\zed$.

In this section we will introduce a new class of maps which are the proper
generalization of an orientation preserving homeomorphism. For completeness we
will recall some related results from \cite{fmot07} where these
maps were first introduced.

\begin{defn}[Degree of a map]
Let $f:U \to \C$ be a map from a simply connected domain $U$ into the plane.
Let $S$ be a positively oriented simple closed curve in $U$, and $w\nin f(S)$ be a
point. Define $f_w:S\to\ucirc$ by \[ f_p(x)=\frac{f(x)-w}{|f(x)-w|}.\]
Then $f_w$ has a well-defined {\em degree} (also known as the {\em winding
number} of $f|_S$ about $w$), denoted $\degree(f_w)=\win(f,S,w)$.
\end{defn}

\begin{defn}
A map $f:U \to \C$ from a simply connected domain $U$ is
\emph{positively oriented} (respectively, {\em negatively oriented}) provided
for each simple closed curve $S$ in $U$ and each point $w\in f(T(S))\setminus
f(S)$, $\degree(f_w)>0$ ($\degree(f_w)<0$, respectively).
\end{defn}

A holomorphic map $f:\C\to\C$ is a prototype of a positively oriented map.
Hence the results obtained in this paper apply to them. However, in general
positively oriented maps do not have to be light (i.e., a positively oriented
map can map a subcontinuum of $\C$ to a point). Observe also, that for points $w\nin T(f(S))$
we have $\degree(f_w)=0$.

A  map $f:\C\to\C$ is \emph{oriented} provided for each simple closed curve $S$
and each $x\in T(S)$, $f(x)\in T(f(S))$. Every positively or negatively
oriented map is oriented (indeed, otherwise there exists $x\in T(S)$ with
$f(x)\nin T(f(S))$ which implies that $\win(f,S,f(x))=0$, a contradiction). A
map $f:\C\to \C$ is \emph{confluent} provided for each subcontinuum
$K\subset\C$ and every component $C$ of $f^{-1}(K)$, $f(C)=K$. It is well known
that both open and monotone maps (and hence compositions of such maps) of
continua are confluent. It follows  from a result of Lelek and Read
\cite{leleread74} that each confluent mapping of the plane is the composition
of a monotone map and a light open map. Theorem~\ref{orient} is obtained in
\cite{fmot07}.

\begin{thm} \label{orient}
Suppose that $f:\C\to\C$ is a surjective map. Then the
following are equivalent:

\begin{enumerate}

\item\label{pnorient} $f$ is either positively or negatively oriented;

\item \label{iorient}$f$ is
oriented;

\item\label{conf} $f$ is confluent.

\end{enumerate}

Moreover, if $f$ satisfies these properties then for any non-separating
continuum $X$ we have $f(\bd(X))\supset \bd(f(X))$.

\end{thm}

Let $X$ be an non-separating plane continuum. A {\em crosscut} of $U=\C\sm X$
is an open arc $A\subset U$ such that $\cl(A)$ is an arc with exactly two
endpoints in $\bd(U)$.  Evidently, a crosscut of $U$ separates $U$ into two
disjoint domains, exactly one unbounded.

Let $S$ be a simple closed curve in $\C$ and suppose $g:S\to\C$ has no fixed
points on $S$. Since $g$ has no fixed points on $S$, the point $z-g(z)$ is
never $0$. Hence the unit vector $v(z)=\frac{g(z)-z}{|g(z)-z|}$ always exists.
Let $z(t)$ be a convenient counterclockwise parameterization of $S$ by $t\in
\uc$ and define the map $\ol{v}=v\circ z:\uc\to \uc$ by

\[\ol{v}(t)=v(z(t))=\frac{g(z(t))-z(t)}{|g(z(t))-z(t)|}.\]

\noindent Then
$\ind{g}{S}$, the \emph{index of $g$ on $S$}, is the {\em degree} of $\ol v$.
The following theorem (see, e.g., \cite{fmot07}) is a major tool in finding
fixed points of continuous maps of the plane.

\begin{thm}\label{basic} Suppose that $S\subset \C$ is a simple closed curve
and $f:T(S)\to \C$ is a continuous map such that $\ind{f}{S}\not=0$.
Then $f$ has a fixed point in $T(S)$.
\end{thm}

Theorem~\ref{basic} applies to Problem~\ref{fpt} as follows. Given a non-separating continuum
$X\subset \C$ one constructs a simple closed curve $S$ approximating $X$ so that
the index of $f$ on $S$ can be computed. If it is not equal to zero, it implies the existence
of a fixed point in $T(S)$, and if $S$ is tight enough, in $X$. Hence our efforts should be
aimed at constructing $S$ and computing $\ind{f}{S}$. One way of doing so is to
use Bell's notion of variation which we will now introduce.

Suppose that $X$ is a non-separating plane continuum and $S$ is a simple closed
curve such that $X\subset T(S)$ and $S\cap X$ consists of more than one point.
Then we will call $S$ a \emph{bumping simple closed curve of $X$}. Any subarc
of $S$, both of whose endpoints are in $X$, is called a \emph{bumping arc of
$X$} or a \emph{link of $S$}. Note that any bumping arc $A$ of $X$ can be
extended to a bumping simple closed curve $S$ of $X$. Hence, every bumping arc
has a natural order $<$ inherited from the positive circular order of a bumping
simple closed curve $S$ containing $A$. If $a<b$ are the endpoints of $A$, then
we will often write $A=[a,b]$.  Also, by the \emph{shadow $Sh(A)$ of $A$}, we
mean the union of all bounded components of $\C\sm (X\cup A)$.

The {\em standard junction} $J_0$ is the union of the three rays
$R_i=\{z\in\C\mid z=re^{i\pi/2}, r\in[0,\infty)\}$, $R_+=\{z\in\C\mid z=re^{0},
r\in[0,\infty)\}$, $R_-=\{z\in\C\mid z=re^{i\pi}, r\in[0,\infty)\}$, having the
origin $0$ in common.  By $U$ we denote the lower half-plane $\{z\in\C\mid
z=x+iy, y<0\}$. A {\em junction} $J_v$ is the image of $J_0$ under any
orientation-preserving homeomorphism $h:\C\to\C$ where $v=h(0)$.  We will often
suppress $h$ and refer to $h(R_i)$ as $R_i$, and similarly for the remaining
rays and the region $h(U)$.

\begin{defn}[Variation on an arc]
Let $f:\C\to\C$ be a map, $X$ be a non-separating plane continuum, $A=[a,b]$ be
a bumping subarc of $X$ with $a<b$, $\{f(a),f(b)\}\subset X$ and $f(A)\cap
A=\0$.  We define the {\em variation of $f$ on $A$ with respect to $X$},
denoted $\var{f}{A}$, by the following algorithm:

\begin{enumerate}

\item Choose an \op\ homeomorphism $h$ of $\C$ such that
$h(0)=v\in A$ and $X\subset h(U)\cup\{v\}$.

\item \label{crossings} {\em Crossings:} Consider the set
$K=[a,b]\cap f^{-1}(J_v)$. Move along $A$ from $a$ to $b$. Each time a point of
$[a,b]\cap f^{-1}(R_+)$ is followed immediately by a point of $[a,b]\cap
f^{-1}(R_i)$ in $K$, count $+1$. Each time a point of $[a,b]\cap
f^{-1}(R_i)$ is followed immediately by a point of $[a,b]\cap
f^{-1}(R_+)$ in $K$, count $-1$. Count no other crossings.

\item The sum of the crossings found above is the
variation, denoted $\var{f}{A}$.

\end{enumerate}

\end{defn}

It is shown in \cite{fmot07} that the variation does not depend on the choice
of a junction satisfying the above listed properties. Informally, one can
understand the notion of variation as follows. Since $f(A)\cap A=\0$, we can
always complete $A$ with another arc $B$ (now connecting $b$ to $a$) to a
simple closed curve $S$ disjoint from $J_v$ so that $v\nin f(S)$. Then it is
easy to see that $\win(f, S, v)$ cab be obtained by summing up $\var{f}{A}$ and
the similar count for the arc $B$ (observe that the latter is \emph{not} the
variation of $B$ because to compute that we will need to use another junction
``based'' at a point of $B$).

Any partition $A=\{a_0<a_1<\dots<a_n<a_{n+1}=a_0\}\subset X\cap S$ of a bumping
simple closed curve $S$ of a non-separating continuum $X$ such that for all
$i$, $f(a_i)\in X$ and $f([a_i,a_{i+1}])\cap [a_i,a_{i+1}]=\0$ is called an
\emph{allowable partition of $S$.} We will call the bumping arcs
$[a_i,a_{i+1}]$ \emph{links (of $S$)}. It is shown in \cite{fmot07}
that the variation of a bumping arc is well-defined. Moreover, it follows from
that paper (see Theorem~2.12 and Remark~2.19) that:

\begin{thm}\label{FMOT} Let $S$ be a simple closed curve, $X=T(S)$ and
let $a_0<a_1<\dots<a_n<a_0=a_{n+1}$ be points in $S$ (in the positive circular
order around $S$) such that for each $i$, $f(a_i)\in T(S)$ and, if
$Q_i=[a_i,a_{i+1}]$, then
$f(Q_i)\cap Q_i=\0$. \\
Then
\[\ind{f}{S}=\sum_{i} \var{f}{Q_i}+1.\]
\end{thm}

Observe that if the points $a_i, i=1, \dots, n$ satisfying the properties of
Theorem~\ref{FMOT} can be chosen then there are no fixed points of $f$ in $S$
and $\ind{f}{S}$ is well-defined. Theorem~\ref{FMOT} shows that if we define
$\var{f}{S}=\sum_i\var{f}{Q_i}$, then $\var{f}{S}$ is well defined and
independent of the choice of the allowable partition of $S$ and of the choice
of the junctions used to compute $\var{f}{Q_i}$.

\begin{lem}\label{closed}
Let $f:\C\to\C$ be a map, $X$ be a non-separating continuum and $C=[a,b]$ be a
bumping arc of $X$ with $a<b$.
%such that $f(X)\subset X$. Suppose $C=(a,b)$ is a cross cut of the continuum $X$.
Let $v\in [a,b]$ be a point and let $J_v$ be a junction such that $J_v\cap
(X\cup C)=\{v\}$. Suppose that $J_v\cap f(X)=\0$. Then there exists an arc $I$
such that $S=I\cup C$ is a bumping simple closed curve of $X$ and
$f(I)\cap J_v=\0$.
\end{lem}

\begin{proof}
Since $f(X)\cap J_v=\0$, it is clear that there exists an arc $I$ with
endpoints $a$ and $b$ near X such that $I\cup C$ is a simple closed curve,
$X\subset T(I\cup C)$ and $f(I)\cap J_v=\0$. This completes the proof.
\end{proof}

The next corollary gives a sufficient condition for the non-negativity
of the variation of an arc.

\begin{cor}\label{posvar} Suppose $f:\C\to\C$ is a positively oriented map,
$C=[a,b]$ is a bumping arc of $X$, $f(C)\cap C=\0$ and
$J_v\subset[\C\sm X]\cup\{v\}$ is a junction with $J_v\cap C=\{v\}$. Suppose
that $f(\{a,b\})\subset X$ and there exists a continuum $K\subset X$ such that
$f(K)\cap J_v\subset \{v\}$. Then $\var{f}{C}\geq 0$.
\end{cor}

\begin{proof}
Observe that since $f(a), f(b)\in X$ then $\var{f}{C}$ is well-defined. Consider a few cases.
Suppose first that $f(K)\cap J_v=\0$. Then, by
Lemma~\ref{closed}, there exists an arc $I$ such that $S=I\cup C$ is a bumping
simple closed curve around $K$ and $f(I)\cap J_v=\0$ (it suffices to choose $I$ very close to $K$).
Then $v\in\C\setminus f(S)$. Hence
$\var{f}{C}=\text{win}(f,S,v)\geq 0$. Suppose next that $f(K)\cap J_v=\{v\}$.
Then we can perturb the junction $J_v$ slightly in a small neighborhood of $v$,
obtaining a new junction $J_d$ such that intersections of $f(C)$ with $J_v$ and
$J_d$ are the same (and, hence, yield the same variation) and $f(K)\cap
J_d=\0$. Now proceed as in the first case.
\end{proof}

\section{Main results}

\subsection{Dendrites}
In this subsection we generalize Theorem~\ref{dendr} to non-invariant
dendrites. We will also show that in certain cases the dendrite must contain
infinitely many periodic cutpoints (recall that if $Y$ is a continuum and $x\in
Y$ then $\val_Y(x)$ is the number of connected components of $Y\sm \{x\}$, and
$x$ is said to be an \emph{endpoint (of $Y$)} if $\val_Y(x)=1$, a
\emph{cutpoint (of $Y$)} if $\val_Y(x)>1$ and a \emph{vertex/branchpoint (of
$Y$)} if $\val_Y(x)>2$). These results have applications in complex dynamics
\cite{bco08}. In this subsection given two points $a, b$ of a dendrite we
denote by $[a, b], (a, b], [a, b), (a, b)$ the unique closed, semi-open and
open arcs connecting $a$ and $b$ in the dendrite. More specifically, unless
otherwise specified, the situation considered in this subsection is as follows:
$D_1\subset D_2$ are dendrites and $f:D_1\to D_2$ is a continuous map. Set
$E=\ol{D_2\sm D_1}\cap D_1$. In other words, $E$ consists of points at which
$D_2$ ``grows'' out of $D_1$. A point $e\in E$ may be an endpoint of $D_1$
(then there is a unique component of $D_2\sm \{e\}$ which meets $D_1$) or a
cutpoint of $D_1$ (then there are several components of $D_2\sm \{e\}$ which
meet $D_1$).

The following theorem is a simple extension of the real result claiming that if
there are points $a<b$ in $\R$ such that $f(a)<a, f(b)>b$ then there exists a
fixed point $c\in (a, b)$ (case (b) described in Introduction).

\begin{prop}\label{fixpt-1}
Suppose that $a, b\in D_1$ are such that $a$ separates $f(a)$ from $b$ and $b$
separates $f(b)$ from $a$. Then there exists a fixed point $c\in (a, b)$ which
is a cutpoint of $D_1$ (and hence $D_2$). In particular if there are two points
$e_1\ne e_2\in E$ such that $f(e_i)$ belongs to a component of $D_2\sm \{e_i\}$
disjoint from $D_1$ then there exists a fixed point $c\in (a, b)$ which
is a cutpoint of $D_1$ (and hence $D_2$).
\end{prop}

\begin{proof}
It follows that we can find a sequence of points $a_{-1}, \dots$ in $(a, b)$
such that $f(a_{-n-1})=a_{-n}$ and $a_{-n-1}$ separates $a_{-n}$ from $b$.
Clearly, $\lim_{n\to \iy} a_{-n}=c\in [a, b]$ is a fixed point as desired. If
there are two points $e_1\ne e_2\in E$ such that $f(e_i)$ belongs to a
component of $D_2\sm \{e_i\}$ disjoint from $D_1$ then the above applies to
them.
\end{proof}

The other real case (case (a)) described in Introduction is somewhat more
difficult to generalize. Definition~\ref{bouscr} extends it (i.e. the real case
when $a<b$ are points of $\R$ such that $f(a)>a$ and $f(b)<b$) onto dendrites.

\begin{defn}\label{bouscr}
Suppose that in the above situation the map $f$ is such that for each non-fixed
point $e\in E$, $f(e)$ is contained in a component of $D_2\sm\{e\}$ which meets
$D_1$. Then we say that $f$ has the \emph{boundary scrambling property} or that
it \emph{scrambles the boundary}. Observe that if $D_1$ \emph{is} invariant
then $f$ automatically scrambles the boundary.
\end{defn}

The next definition presents a useful topological version of repelling at a
fixed point.

\begin{defn}\label{wkrep}
Suppose that $a\in D_1$ is a fixed point and that there exists a component $B$
of $D_1\sm \{a\}$ such that arbitrarily close to $a$ in $B$ there exist fixed
cutpoints of $f$ or points $x$ separating $p$ from $f(x)$. Then say that $a$ is
a \emph{weakly repelling fixed point (of $f$ in $B$)}. A periodic point $a$ is
said to be \emph{weakly repelling} if there exists $n$ and a component $B$ of
$D_1\sm \{a\}$ such that $a$ is a weakly repelling fixed point of $f^n$ in $B$.
\end{defn}

It is easy to see that a fixed point $a$ is weakly repelling in $B$ if and only
if either $a$ is a limit of fixed cutpoints of $f$ in $B$, or there exist a
neighborhood $U$ of $a$ in $\ol{B}$ and a point $x\in U\sm \{a\}$ such that $U$
contains no fixed points but $a$ and $x$ separates $p$ from $f(x)$. Indeed, in
the latter case by continuity there exists a point $x_1\in (a, x)$ such that
$f(x_{-1})=x$ and this sequence of preimages can be extended towards $a$ inside
$(a, x)$ so that it converges to $a$ (otherwise it would converge to a fixed
point inside $U$ distinct from $a$, a contradiction). In particular, \emph{if
$a$ is a weakly repelling fixed point of $f$ then $a$ is a weakly repelling
fixed point of $f^n$ for any $n$}. Moreover, since there are only countably
many vertices of $D_2$ and their images under $f$ and its powers, we can choose
$x$ and its backward orbit converging to $a$ so that all its points are
cutpoints of $D_2$ of valence $2$. From now on we assume that to each weakly
repelling fixed point $a$ of $f$ in $B$ which is not a limit point of fixed
cutpoints in $B$ we associate a point $x_a=x\in B$ of valence $2$ in $B$
separating $a$ from $f(x)$ and a neighborhood $U_a=U\subset \ol{B}$ which is
the component of $\ol{B}\sm \{x\}$ containing $a$.

As an important tool we will need the following retraction closely related to
the described above situation.

\begin{defn}\label{retr}
For each $x\in D_2$ there exists a unique arc (possibly a point) $[x,d_x]$ such
that $[x,d_x]\cap D_1=\{d_x\}$. Hence there exists a natural monotone
retraction $r:D_2\to D_1$ defined by $r(x)=d_x$, and the map $g=g_f=r\circ
f:D_1\to D_1$ which is a continuous map of $D_1$ into itself. We call the map
$r$ the \emph{natural retraction (of $D_2$ onto $D_1$)} and the map $g$ the
\emph{retracted (version of) $f$}.
\end{defn}

The map $g$ is designed to make $D_1$ invariant so that Theorem~\ref{dendr}
applies to $g$ and allows us to conclude that there are $g$-fixed points.
However these points are not necessarily fixed points of $f$. Indeed, $g(x)=x$
means that $r\circ f(x)=x$. Hence it really means that $f$ maps $x$ to a point
belonging to a component of $D_2\sm D_1$ which ``grows'' out of $D_1$ at $x$.
In particular, it means that $x\in E$. Thus, if $g(x)=x$ and $x\notin E$ then
$f(x)=x$. In general, it follows from the construction that if $f(x)\ne g(x)$,
then $g(x)\in E$ because points of $E$ are exactly those points of $D_1$ to
which points of $D_2\sm D_1$ map under $r$. We are ready to prove our first
lemma in this direction.

\begin{prop}\label{fixpt0}
Suppose that $f$ scrambles the boundary. Then $f$ has a fixed point.
\end{prop}

\begin{proof}
We may assume that there are no $f$-fixed points $e\in E$. By
Theorem~\ref{dendr} the map $g_f=g$ has a fixed point $p\in D_1$. It follows
from the fact that $f$ scrambles the boundary that points of $E$ are not
$g$-fixed. Hence $p\notin E$, and by the argument right before the lemma
$f(p)=p$ as desired.
\end{proof}

It follows from Proposition~\ref{fixpt-1} and Proposition~\ref{fixpt0} that the
only behavior of points in $E$ which does not force the existence of a fixed
point in $D_1$ is when exactly one point $e\in E$ maps into a component of
$D_2\sm \{e\}$ which is disjoint from $D_1$ whereas any other point $e'\in E$
maps into a component of $D_2\sm \{e'\}$ which is not disjoint from $D_1$.

The next lemma shows that in some cases $p$ can be chosen to be a cutpoint of
$D_1$.

\begin{lem}\label{fxctpt}
Suppose that $f$ scrambles the boundary and all $f$-fixed endpoints of $D_1$
are weakly repelling. Then there is a fixed cutpoint of $f$.
\end{lem}

\begin{proof}
Suppose that $f$ has no fixed cutpoints. By Proposition~\ref{fixpt0}, the set
of fixed points of $f$ is not empty. Hence we may assume that \emph{all} fixed
points of $f$ are endpoints of $D_1$. Let $a, b$ be distinct fixed points of
$f$. Then it follows that either $U_a\subset U_b$, or $U_b\subset U_a$, or
$U_a\cap U_b=\0$. Indeed, suppose that $x_a\in U_b$. Let us show that then
$U_a\subset U_b$. Indeed, otherwise by Lemma~\ref{fixpt0}.(1) there exists a
fixed point $c\in (x_a, x_b)$, a contradiction. Now, suppose that neither
$x_a\in U_b$ nor $x_b\in U_a$. Then clearly $U_a\cap U_b=\0$. Consider an open
covering of the set of all fixed points $a\in D_1$ by their neighborhoods $U_a$
and choose a finite subcover. By the above we may assume that its consists of
pairwise disjoint sets $U_{a_1}, \dots, U_{a_k}$. Consider now the component
$Q$ of $D_1$ whose endpoints are the points $a_1, \dots, a_k$. It is easy to
see that $f|_Q$, with $Q$ considered as a subdendrite of $D_2$, scrambles the
boundary and has no fixed endpoints. Hence an $f$-fixed point $p\in Q$,
existing by Lemma~\ref{fixpt0}, must be a cutpoint of $D_1$.
\end{proof}

Lemma~\ref{fxctpt} is helpful in the next theorem which shows that under some rather
weak assumptions on periodic points the map has infinitely many periodic cutpoints.

\begin{thm}\label{infprpt}
Suppose that $f:D\to D$ is continuous and all its periodic points are weakly repelling.
Then $f$ has infinitely many periodic cutpoints.
\end{thm}

\begin{proof}
By way of contradiction suppose that there are finitely many periodic cutpoints
of $f$. Without loss of generality we may assume that these are points $a^1,
\dots, a^k$ each of which is \emph{fixed} under $f$. Let $A=\cup^k_{i=1} a^i$
and let $B$ be component of $D\sm A$. Then $\ol{B}$ is a subdendrite of $D$ to
which the above tools apply: $D$ plays the role of $D_2$, $\ol{B}$ plays the
role of $D_1$, and $E$ is exactly the boundary $\bd(B)$ of $B$ (by the
construction $\bd(B)\subset A$). Suppose that each point $a\in \bd(B)$ is
weakly repelling in $B$. Then by the assumptions of the theorem
Lemma~\ref{fxctpt} applies to this situation. It follows that there exists a
fixed cutpoint of $B$, a contradiction. Hence for some $a\in \bd{B}$ we have
that $a$ is \emph{not} weakly repelling in $B$. Since by the assumptions $a$
\emph{is} weakly repelling, there exists another component, say, $C$, of $D\sm
A$ such that $a\in \bd(C)$ and $a$ \emph{is} weakly repelling in $C$.

We can now apply the same argument to $C$. If all boundary points of $C$ are
weakly repelling in $C$ then by Lemma~\ref{fxctpt} $C$ will contain a fixed
cutpoint, a contradiction. Hence there exists a point $d\in A$ such that $d$ is
\emph{not} weakly repelling in $C$ and a component $F$ of $D\sm A$ whose
closure meets $\ol{C}$ at $d$, and $d$ \emph{is} weakly repelling in $F$.
Clearly, after finitely many steps this process will have to end ultimately
leading to a component $Z$ of $D\sm A$ such that all points of $\bd(Z)$ are
weakly repelling in $Z$. Since here the set $\bd(Z)$ plays the role of the set
$E$ from above and by the assumptions of the theorem we see that
Lemma~\ref{fxctpt} applies to $Z$ and there exists a fixed cutpoint of $Z$, a
contradiction.
\end{proof}

An important application of Theorem~\ref{infprpt} is to the dendritic
\emph{topological Julia sets}. They can be defined as follows. Consider an
equivalence relation $\sim$ on the unit circle $\ucirc$. Equivalence classes of
$\sim$ will be called \emph{($\sim$-)classes} and will be denoted by boldface
letters. A $\sim$-class consisting of two points is called a \emph{leaf}; a
class consisting of at least three points is called a \emph{gap} (this is more
restrictive than Thurston's definition in \cite{thur85}; we follow \cite{bl02}
in our presentation).  Fix an integer $d>1$ and denote the map $z^d:\uc\to\uc$
by $\si_d$. Then $\sim$ is said to be a \emph{($d$-)invariant lamination} if:

\noindent (E1) $\sim$ is \emph{closed}: the graph of $\sim$ is a closed set in
$\ucirc \times \ucirc$;

\noindent (E2) $\sim$ defines a \emph{lamination}, i.e., it is \emph{unlinked}:
if $\g_1$ and $\g_2$ are distinct $\sim$-classes, then their convex hulls
$\ch(\g_1), \ch(\g_2)$ in the unit disk $\disk$ are disjoint,

\noindent (D1) $\sim$ is \emph{forward invariant}: for a class $\g$, the set
$\si_d(\g)$ is a class too

\noindent which implies that

\noindent (D2) $\sim$ is \emph{backward invariant}: for a class $\g$, its
preimage $\si_d^{-1}(\g)=\{x\in \ucirc: \si_d(x)\in \g\}$ is a union of
classes;

\noindent (D3) for any gap $\g$, the map $\si_d|_{\g}: \g\to \si_d(\g)$ is a
\emph{covering map with positive orientation}, i.e., for every connected
component $(s, t)$ of $\ucirc\setminus \g$ the arc $(\si_d(s), \si_d(t))$ is a
connected component of $\ucirc\setminus \si_d(\g)$.

The lamination in which all points of $\uc$ are equivalent is said to
\emph{degenerate}. It is easy to see that if a forward invariant lamination
$\sim$ has a class with non-empty interior then $\sim$ is degenerate. Hence
equivalence classes of any non-degenerate forward invariant lamination are
totally disconnected.

Call a class $\g$ \emph{critical} if $\si_d|_{\g}: \g\to \si(\g)$ is not
one-to-one, \emph{(pre)critical} if $\si_d^j(\g)$ is critical for some $j\ge
0$, and \emph{(pre)periodic} if $\si^i_d(\g)=\si^j_d(\g)$ for some $0\le i<j$.
A gap $\g$ is \emph{wandering} if $\g$ is neither (pre)periodic nor
(pre)critical. Let $p: \ucirc\to J_\sim=\ucirc/\sim$ be the quotient map of
$\ucirc$ onto its quotient space $J_\sim$, let $f_\sim:J_\sim \to J_\sim$ be
the  map induced by $\sigma_d$. We call $J_\sim$ a \emph{topological Julia set}
and the induced map $f_\sim$ a \emph{topological polynomial}. It is easy to see
that if $\g$ is a $\sim$-class then $\val_{J_\sim}(p(\g))=|\g|$ where by $|A|$
we denote the cardinality of a set $A$.

\begin{thm}\label{lamwkrp}
Suppose that the topological Julia set $J_\sim$ is a dendrite and
$f_\sim:J_\sim\to J_\sim$ is a topological polynomial. Then all periodic points
of $f_\sim$ are weakly repelling and $f_\sim$ has infinitely many periodic
cutpoints.
\end{thm}

\begin{proof}
Suppose that $x$ is an $f_\sim$-fixed point and set $\g=p^{-1}(x)$. If $x$ is
an endpoint then $\g$ is a singleton. Connect $x$ with a point $y\ne x$. Then
the arc $[x, y]\subset J_\sim$ contains points $y_k\to x$ of valence $2$
because, as is well known, there are no more than countably many vertices of
$J_\sim$. It follows that $\sim$-classes $p^{_1}(y_k)$ are leaves separating
$\g$ from the rest of the circle and repelled from $\g$ under the action of
$\si$. Hence $f_\sim(y_i)$ is separated from $x$ by $y_i$ and so $x$ is weakly
repelling.

Suppose that $x$ is not an endpoint. Choose a very small connected neighborhood
$U$ of $x$. It is easy to see that each component $A$ of $U\sm \{x\}$
corresponds to a unique chord $\ell_A\in \bd(\ch(\g))$. Moreover, for each
component $A$ of $U\sm \{x\}$ there exists a unique component $B$ of $U\sm
\{x\}$ such that $f_\sim(A)\cap B\ne \0$. Hence there is a map $h$ from the set
$\A$ of all components of $U\sm \{x\}$ to itself. Suppose that there exists
$E\in \A$ and $n>0$ such that $h^n(E)=E$. Then it follows that the endpoints of
$\ell_E$ are fixed under $\si^n$. Connect $x$ with a point $y\in E$ and choose,
as in the previous paragraph, a sequence of points $y_k\in [x, y], y_k\to x$ of
valence $2$. Then again by the repelling properties of $\si^n$ it follows that
$f_\sim(y_i)$ is separated from $x$ by $y_i$ and so $x$ is weakly repelling
(for $f^m_\sim$ in $E$).

It remains to show that there $E\in \A$ with $h^n(E)=E$ for some $n>0$ must
exist. Suppose otherwise. To each component $C$ of $U\sm \{x\}$ we associate
the corresponding component $J_C$ of $J_\sim\sm \{a\}$ containing $C$. Then
there are only finitely many such components $C$ of $U\sm \{x\}$ that $J_C$
contains a critical point; denote their collection by $\mathcal C$. Let us show
that an eventual $h$-image of every $E\in \A$ must coincide with an element of
$\mathcal C$. Indeed, otherwise there is a component $E\in \A$ such that all
$h^k(E)$ are distinct and the map $f_\sim|_{J_{h^k(E)}}$ is a homeomorphism.
Clearly, this implies the existence of a wandering subcontinuum $K$ of
$J_\sim$. However by Theorem C \cite{bl02} this is impossible.

Hence all periodic points of $f_\sim$ are weakly repelling and by
Theorem~\ref{infprpt} $f_\sim$ has infinitely many periodic cutpoints.
\end{proof}

\subsection{Positively oriented maps of the plane}

In this subsection we will first obtain a general fixed point theorem which
shows that if a non-separating plane continuum, not necessarily invariant, maps
in an appropriate way, then it must contain a fixed point. This extends
Theorem~\ref{fmot1}. Let us denote the family of all positively oriented maps
of the plane by $\p$.

\begin{thm}[\cite{fmot07}]\label{fmot1}
Suppose that $f\in \p$ and that $X\subset \C$ is a non-separating continuum
such that $f(X)\subset X$. Then there exists a fixed point $p\in X$.
\end{thm}

To proceed we will need to generalize Corollary~\ref{posvar} to a more general
situation. To this end we introduce a definition similar to the one given for
dendrites in the previous section.

\begin{defn}\label{scracon}
Suppose that $f\in \p$ and $X$ is non-separating continuum. Suppose that there
exist $n\ge 0$ disjoint non-separating continua $Z_i$ such that the following
properties hold:

\begin{enumerate}

\item $f(X)\sm X\subset \cup_i Z_i$;

\item for all $i$,  $Z_i\cap X=K_i$ is a non-separating continuum;

%\item
%for each component $C$ of $f(X)\sm X$, there exists a unique
%$i(C)\in\{1,\dots,n\}$ such that $\ol{C}\cap X\subset K_{i(C)}$,

\item for all $i$, $f(K_i)\cap [Z_i\sm K_i]=\0$.

\end{enumerate}

\noindent Then the map $f$ is said to \emph{scramble the boundary (of $X$}). If
instead of (3) we have

\begin{enumerate}

\item[(3a)] for all $i$, either $f(K_i)\subset K_i$, or
$f(K_i)\cap Z_i=\0$

\end{enumerate}

\noindent then we say that $f$  \emph{strongly scrambles the boundary (of
$X$)}. In either case, $K_i$'s are called \emph{exit continua (of $X$)}. Note
that since $Z_i\cap X$ is a continuum, $X\cup (\bigcup Z_i)$ is a
non-separating continuum. Speaking of maps which (strongly) scramble the
boundary, we always use the notation from this definition unless explicitly
stated otherwise.
\end{defn}

Observe that in the situation of Definition~\ref{scracon} if $X$ is invariant
then $f$ automatically strongly scrambles the boundary because we can simply
take  the set of exit continua to be empty. Also, if $f$ strongly scrambles the
boundary of $X$ and $f(K_i)\not\subset K_i$ for any $i$, then it is easy to see
that there exists $\e>0$ such that for every point $x\in X$ either $d(x,
Z_i)>\e$, or $d(f(x),Z_i)>\e$. Let us now prove the following technical lemma.

\begin{lem}\label{posvar+}
Suppose that $f:\C\to\C$ scrambles the boundary of $X$. Let $Q$ be a bumping
arc of $X$ with endpoints $a<b\in X$ such that $f(\{a,b\})\subset X$ and
$f(Q)\cap Q=\0$. Then $\var{f}{Q}\ge 0$.
\end{lem}

\begin{proof}
We will use the notation as specified in the lemma. Suppose first that $Q\sm
\bigcup Z_i\ne \0$ and choose $v\in Q\sm \bigcup Z_i$.
 Since $v\in Q\sm \bigcup Z_i$ and $X\cup (\bigcup
Z_i)$ is non-separating, there exists a junction $J_v$, with $v\in Q$, such
that $J_v\cap[X\cup Q\cup \bigcup Z_i]= \{v\}$ and, hence,  $J_v\cap f(X)\subset\{v\}$. Now the
desired result follows from Corollary~\ref{posvar}.

Observe that if $Q\sm \cup Z_i=\0$ then $Q\subset Z_i$ for some $i$ and so
$Q\cap X\subset K_i$. In particular, both endpoints $a,b$ of $Q$ are contained
in $K_i$. Choose a point $v\in Q$. Then again there is a junction connecting
$v$ and infinity outside $X$ (except possibly for $v$). Since all sets $Z_j,
j\ne i$ are positively distant from $v$ and $X\cup (\bigcup_{i\ne j} Z_i)$ is
non-separating, the junction $J_v$ can be chosen to avoid all sets $Z_j, j\ne
i$. Now, by (3) $f(K_i)\cap J_v\subset \{v\}$, hence by Lemma~\ref{posvar}
$\var{f}{Q}\ge 0$.
\end{proof}

Lemma~\ref{posvar+} is applied in Theorem~\ref{fixpt} in which we show that a
map which strongly scrambles the boundary has fixed points. In fact, it is a
major technical tool in our other results too. Indeed, if we can construct a
bumping simple closed curve $S$ around $X$ such that the endpoints of its
links map back into $X$ while these links move completely off themselves, the
lemma would imply that the variation of $S$ is non-negative. By
Theorem~\ref{FMOT} this would imply that the index of $S$ is positive. Hence by
Theorem~\ref{basic} there are fixed points in $T(S)$. Choosing $S$ to be
sufficiently tight we see that there are fixed points in $X$.

For the sake of convenience we now sketch the proof of Theorem~\ref{fixpt}
which allows us to emphasize the main ideas rather than details. The main steps
in constructing $S$ are as follows. First we assume by way of contradiction
that a map $f:\C\to \C$ has no fixed points in $X$. Then by Theorem~\ref{fmot1}
it implies that $f(X)\not\subset X$ and that $f(K_i)\not\subset K_i$ for any
$i$. By the definition of strong scrambling then $f(K_i)$ is far away from
$Z_i$ for any $i$. Now, since there are no fixed points in $X$ we can choose
the links in $S$ to be very small so that they will all move off themselves.
However some of them will have endpoints mapping outside $X$ which prevents us
from directly applying Lemma~\ref{posvar+} to them. These links will be
enlarged by concatenating them so that the images of the endpoints of these
concatenations are inside $X$ \emph{and} these concatenations still map off
themselves. The bumping simple closed curve $S$ then remains as before, however
the representation of $S$ as the union of links changes because we enlarge some
of them. Still, the construction shows that Lemma~\ref{posvar+} applies to the
new ``bigger'' links and as before this implies the existence of a fixed point
in $X$.

To achieve the goal of replacing some links in $S$ by their concatenations we
consider the links which are mapped outside $X$ in detail using the fact that
$f$ strongly scrambles the boundary (indeed, all other links are such that
Lemma~\ref{posvar+} already applies to them). The idea is to consider the links
of $S$ whose concatenation is a connected piece of $S$ mapping into one $Z_i$.
Then if we begin the concatenation right before the images of links enter $Z_i$
and stop it right after the images of the links exit $Z_i$ we will have one
condition of Lemma~\ref{posvar+} satisfied because the endpoints of the thus
constructed new ``big'' concatenation link $T$ of $S$ map into $X$.

We now need to verify that $T$ moves off itself under $f$. Indeed, this is easy
to see for the end-links of $T$: each end-link has the image ``crossing'' into
$Z_i$ from $X\sm Z_i$, hence the images of end-links are close to $K_i$.
However the sets $K_i$ are mapped far away from $Z_i$ by the definition of
strong scrambling and because none of $K_j$'s is invariant by the assumption.
This implies that the end-links themselves must be far away from $Z_i$. If now
we move from link to link inside $T$ we see that those links cannot approach
$Z_i$ too closely because if they do they will have to ``cross over $K_i$''
into $Z_i$, and then their images will have to be close to the image of $K_i$
which is far away from $Z_i$, a contradiction with the fact that all links in $T$
have endpoints which map  into $Z_i$. In other words, the dynamics of $K_i$ prevents the new
bigger links from getting even close to $Z_i$ under $f$ which shows that they
move off themselves as desired. As before, we now apply Theorem~\ref{FMOT} to
see that $\ind{f}{S}=\var{f}{S}+1$ and then Theorem~\ref{basic} to see that
this implies the existence of a fixed point in $X$.

Given a compact set $K$ denote by $B(K, \e)$ the open set of all points
whose distance to $K$ is less than $\e$.

\begin{thm} \label{fixpt}
Suppose that $f:\C\to\C$ strongly scrambles the boundary of $X$.
Then $f$ has a fixed point in $X$.
\end{thm}

\begin{proof}
If $f(X)\subset X$ then the result follows from \cite{fmot07}. Similarly, if
there exists $i$ such that $f(K_i)\subset K_i$, then $f$ has a fixed point in
$T(K_i)\subset X$ and we are also done. Hence we may assume $f(X)\sm X\ne\0$
and $f(K_i)\cap Z_i=\0$ for all $i$. Suppose that $f$ is fixed point free. Then
there exists $\e>0$ such that for all $x\in X$, $d(x,f(x))>\e$. We may assume
that $\e<\min\{d(Z_i,Z_j)\mid i\ne j\}$. We now choose constants $\eta', \eta,
\da$ and a bumping simple closed curve $S$ of $X$ so that the following holds.

\begin{enumerate}

\item $0<\eta'<\eta<\da<\e/3$.

\item
For each $x\in X\cap B(K_i,3\da)$, $d(f(x), Z_i)>3\da$.

\item
For each $x\in X\sm B(K_i,3\da)$, $d(x,Z_i)>3\eta$.

\item For each $i$ there exists a point $x_i\in X$ such that $f(x_i)=z_i\in Z_i$ and $d(z_i, X)>3\eta$.

\item $X\subset T(S)$ and $A=X\cap S=\{a_0<a_1<\dots<a_n<a_{n+1}=a_0\}$ in the
positive circular order around $S$.

\item $f|_{T(S)}$ is fixed point free.

\item For the closure $Q_i=[a_i, a_{i+1}]$ of a component of $S\sm X$,
$\dia(Q)+\dia(f(Q))<\eta$.

\item For any two points $x, y\in X$ with $d(x, y)<\eta'$ we have $d(f(x), f(y))<\eta$.

\item $A$ is an $\eta'$-net in $\bd(X)$.

\end{enumerate}

Observe that by the triangle inequality, $Q_i\cap f(Q_i)=\0$ for every $i$.

\smallskip

\noindent \textbf{Claim 1.} \emph{There exists a point $a_j$ such that
$f(a_j)\in X\sm \cup \ol{B(Z_i, \eta)}$.}

\noindent \emph{Proof of Claim 1}. Set $\ol{B(Z_i, 3\eta)}=T_i$ and show that
there exists a point $x\in \bd(X)$ with $f(x)\in X\sm \cup T_i$. Indeed,
suppose first that $n=1$. Then $f(K_1)\subset X\sm T_1$ and we can choose any
point of $K_1\cap \bd(X)$ as $x$. Now, suppose that $n\ge 2$. Observe that the
sets $T_i$ are pairwise disjoint compacta. By Theorem~\ref{orient}
$f(\bd(X))\supset \bd(f(X))$. Hence there are points $x_1\ne x_2$ in $\bd(X)$
such that $f(x_1)\in Z_1\subset T_1, f(x_2)\in Z_2\subset T_2$. Since the sets
$f^{-1}(T_i)\cap X$ are pairwise disjoint non-empty compacta we see that the
set $V=\bd(X)\sm \cup f^{-1}(T_i)$ is non-empty (because $\bd(X)$ is a
continuum). Now we can choose any point of $V$ as $x$.

It remains to notice that by the choice of $A$ we can find a point $a_j$ such
that $d(a_j, x)<\eta'$ which implies that $d(f(a_j), f(x))<\eta$ and hence
$f(a_j)\in X\sm \cup \ol{B(Z_i, \eta)}$ as desired.\qed

There exists a point $x_1$ such that $f(x_1)=z_1$ is more than $3\eta$-distant
from $X$. We can find $a\in A$ such that $d(a, x_1)<\eta'$ and hence
$f(a)\nin X$ is at least $2\eta$-distant from $X$. On the other hand,
by Claim 1 there are points of $A$ mapped into $X$. Let $<$ denote the circular order on the set $\{0,1,\dots,n+1\}$
defined by $i<j$ if $a_i<a_j$ in the positive circular order around $S$. Then we can find
$n(1)<m(1)$ such that the following claims hold.

\begin{enumerate}

\item $f(a_{n(1)-1})\in X\sm \cup Z_i$.

\item $f(a_r)\in f(X)\sm X$ for all $r$ with $n(1)\le r\le m(1)-1$
(and so, since $\dia(f(Q_u))<\e/3$ for any $u$ and
$d(Z_s,Z_t)>\e$ for all $s\ne t$, there exists $i(1)$ with
$f(a_r)\in Z_{i(1)}$ for all $n(1)\le r<m(1)$).

\item $f(a_{m(1)}\in X\sm \cup Z_i$.

\end{enumerate}

Consider the arc $Q'_1=[a_{n(1)-1},a_{m(1)}]\subset S$ and show that
$f(Q'_1)\cap Q'_1=\0$. As we walk along $Q'_1$, we begin outside $Z_{i(1)}$
at $f(a_{n(1)-1})$, then enter $Z_{i(1)}$ and walk inside it, and then exit
$Z_{i(1)}$ at $f(a_{m(1)})$. Since every step in this walk is rather short
($\dia(Q_i)+\dia(f(Q_i))<\eta$), we see that $d(f(a_{n(1)-1}), Z_{i(1)})<\eta$
and $d(f(a_{m(1)}), Z_{i(1)})<\eta$. On the other hand for each $r, n(1)\le r<m(1)$ we have
$f(a_r)\in Z_{i(1)}$, hence we see that $d(f(a_r), Z_{i(1)})<\eta$
for each $r, n(1)\le r<m(1)$.
This implies that $d(a_r, K_{i(1)})>3\delta$ (because otherwise $f(a_r)$
would be farther away from $Z_{i(1)}$) and so $d(a_r), Z_{i(1)})>3\eta$
(because $a_r\in X\sm B(K_{i(1)}, 3\delta)$). Since $\dia(Q)+\dia(f(Q))<\eta$,
then $d(Q'_1, Z_{i(1)})>2\da>2\eta$. On the other hand, $d(f(a_r), Z_{i(1)})<\eta$
similarly implies that $d(f(Q'_1), Z_{i(1)})<2\eta$. Thus indeed $f(Q'_1)\cap Q'_1=\0$.

This allows us to replace the original division of $S$ into its prime links
$Q_1, \dots, Q_n$ by a new one in which $Q'_1$ plays the role of a new prime link;
in other words, we simply delete the points $\{a_{n(1)},\dots, a_{m(1)-1}\}$ from
$A$. Continuing in the same manner and moving along $S$, in the end we obtain a finite set
$A'=\{a_0=a'_0<a'_1<\dots<a'_k\}$ such that for each $i$ we have $f(a'_i)\in X\subset
T(S)$ and for each arc $Q'_i=[a'_i,a'_{i+1}]$ we have $f(Q'_i)\cap Q'_i=\0$. Hence, by
Theorem~\ref{FMOT}, $\ind{f}{S}=\sum_{Q'_{i}} \var{f}{Q'_i}+1$. Since by
Lemma~\ref{posvar+}, $\var{f}{Q'_i}\ge 0$ for all $i$, $\ind{f}{S}\ge 1$
contradicting the fact that $f$ is fixed point free in $T(S)$.
\end{proof}

\subsection{Maps with isolated fixed points}

Now we consider maps $f\in \p$ with isolated fixed points; denote the set of
such maps by $\p_i$. We need a few definitions.

\begin{defn}\label{crit}
Given a map $f:X\to Y$ we say that $c\in X$ is a \emph{critical point of $f$}
if for any neighborhood $U$ of $c$, there exist $x_1\ne x_2\in U$ such that
$f(x_1)=f(x_2)$. Hence, if $x$ is not a critical point of $f$, then $f$ is
locally one-to-one near $x$.
\end{defn}

If a point $x$ belongs to a continuum collapsed under $f$ then $x$ is critical;
also any point which is an accumulation point of collapsing continua is
critical. However in this case the map around $x$ may be monotone. A more
interesting case is when the map around $x$ is not monotone; then $x$ is a
\emph{branchpoint} of $f$ and it is critical even if there are no collapsing
continua close by. One can define the \emph{local degree} $\deg_f(a)$ as the
number of components of $f^{-1}(y)$, for a point $y$ close to $f(a)$, which are non-disjoint from
a small neighborhood of $a$. Then branchpoints are exactly the points at which
the local degree is more than $1$. Notice that since we do not assume any
smoothness, a critical point may well be both fixed (periodic) and be such that
small neighborhoods  of $c=f(c)$ map over themselves by $f$.

The next definition is closely related to that of the index of the map on
a simple closed curve.

\begin{defn}\label{indpt}
Suppose that $x$ is a fixed point of a map $f\in \p_i$.
%such that $x$ is not a
%critical point of $f$.
Then the \emph{local index of $f$ at $x$}, denoted by
$\ind{f}{x}$, is defined as $\ind{f}{S}$
%the winding number of a
where $S$ is a small simple closed curve around $x$.
\end{defn}

It is easy to see that if $f\in \p_i$, then the local
index is well-defined, i.e. does not depend on the choice of $S$. By modifying
a translation map one can give an example of a homeomorphism of the plane which
has exactly one fixed point $x$ with local index $0$. Still in some cases the
local index at a fixed point must be positive.

\begin{defn}\label{toprepat}
A fixed point $x$ is said to be \emph{topologically repelling} if there exists
a sequence of simple closed curves $S_i\to \{x\}$ such that $x\in
\Int(T(S_i))\subset T(S_i)\subset \Int(T(f(S_i))$. A fixed point $x$ is said to
be \emph{topologically attracting} if there exists a sequence of simple closed
curves $S_i\to \{x\}$ not containing $x$ and such that $x\in
\Int(T(f(S_i))\subset T(f(S_i))\subset \Int(T(S_i))$.
\end{defn}

\begin{lem}\label{ind1} If $a$ is a topologically repelling fixed point
then $\ind{f}{a}=\deg_f(a)\ge 1$ where $d$ is the local degree. If however $a$ is a
topologically attracting fixed point then $\ind{f}{a}=1$.
\end{lem}

\begin{proof}
Consider the case of the repelling fixed point $a$. Then it follows that, as
$x$ runs along a small simple closed curve $S$ with $a\in T(S)$, the vector
from $x$ to $f(x)$ produces the same winding number as the vector from $a$ to
$f(x)$, and it is easy to see that the latter equals $\deg_f(a)$. The argument
with attracting fixed point is similar.
\end{proof}

If however $x$ is neither topologically repelling nor topologically attracting,
then $\ind{f}{x}$ could be greater than $1$ even in the non-critical case.
Indeed, take a neutral fixed point of a rational function. Then it follows that
if $f'(x)\ne 1$ then $\ind{f}{x}=1$ while if $f'(x)=1$ then $\ind{f}{x}$ is the
multiplicity at $x$ (i.e., the local degree of the map $f(z)-z$ at $x$). This
is related the following useful theorem. It is a version a more general,
topological \emph{argument principle} stated in the convenient for us form.

\begin{thm}\label{argupr}
Suppose that $f\in \p_i$. Then for any simple closed curve $S\subset\C$ which
contains no fixed points of $f$ its index equals the sum of local indices taken
over all fixed points in $T(S)$.
\end{thm}

Theorem~\ref{argupr} implies Theorem~\ref{basic} but provides more information.
In particular if $S$ were a simple closed curve and if we knew that the local
index at any fixed point $a\in T(S)$ is $1$, it would imply that $\ind{f}{S}$
equals the number $n(f, S)$ of fixed points of $f$ in $T(S)$. By the above
analysis this holds if all $f$-fixed points in $T(S)$ are either repelling, or
attracting, or neutral and such that $f$ has a complex derivative $f'$ in a
small neighborhood of $x$, and $f'(x)\ne 1$.

In the spirit of the previous parts of the paper, we are still concerned with
finding $f$-fixed points inside non-invariant continua of which $f$ (strongly)
scrambles the boundary. However we now specify the types of fixed points we are
looking for. Thus, the main result of this subsection proves the existence of
specific fixed points in non-degenerate continua satisfying the appropriate
boundary conditions and shows that in some cases such continua must be
degenerate. It is in this latter form that we apply the result later on in
Section 4.

Given a non-separating continuum $X$,  a ray $R\subset \C\sm X$ from $\infty$ which lands on
$x\in X$ (i.e., $\ol{R}\sm R=\{x\}$) and a crosscut $Q$ of $X$ we say that $Q$ and $R$ \emph{cross essentially}
provided there exists $r\in R$ such that the subarc $[x,r]\subset \ol{R}$ is contained in $\sh(Q)$.
The next definition complements the previous one.

\begin{defn}\label{repout}
If $f(p)=p$ and $p\in \bd(X)$ then we say that $f$ \emph{repels outside $X$ at
$p$} provided there exists a ray $R\subset \C\sm X$ from $\infty$ which lands on $p$
and a sequence of simple closed curves $S^j$ bounding
closed disks $D^j$ such that $D^1\supset D^2\supset\dots, $  $\cap D^j=\{p\}$,
$f(D^1\cap X)\subset X$, $f(\ol{S^j\sm X})\cap D^j=\0$ and for each $j$ there
exists a component $Q^j$ of $S^j\sm X$ such that $Q^j\cap R\ne\0$ and  $\var{f}{Q^j}\ne 0$. If $f\in
\p$  and $f$ scrambles the boundary of $X$, then by Lemma~\ref{posvar+}, for any
component $Q$ of $S^j\sm X$ we have $\var{f}{Q}\ge 0$ so that in this case
$\var{f}{Q^j}>0$.
\end{defn}

% in the final version we should prove that this is equivalent to
% just one prime end with repelled cross-cuts of variation >0 around it.
% the corresponding theorem  in FMOT deals with the case when there is no
% fixed pt in the shadow of the crosscut

The next theorem is the main result of this subsection.

\begin{thm}\label{locrot}
Suppose that $f\in \p_i$, and $X\subset\C$ is a non-separating continuum or a
point. Moreover, the following conditions hold.

\begin{enumerate}

\item\label{fV}
For each fixed point $p\in X$ we have that $\ind{f}{x}=1$ and $f$ repels
outside $X$ at $p$.

\item
The map $f$ scrambles the boundary of $X$. Moreover, either $f(K_i)\cap
Z_i=\0$, or there exists a neighborhood $W_i$ of $K_i$ with $f(W_i\cap
X)\subset X$.

\end{enumerate}

Then $X$ is a point.
\end{thm}

\begin{proof}
Suppose that $X$ is not a point. Since $f\in \p_i$, there exists a simply
connected neighborhood $V$ of $X$ such that all fixed points
$\{p_1,\dots,p_m\}$ of $f|_{\ol{V}}$ are contained in $X$. We will show that
then $f$ must have at least $m+1$ fixed points in $V$, a contradiction. The
proof will proceed like the proof of Theorem~\ref{fixpt}: we construct a tight
bumping simple closed curve $S$ such that $X\subset T(S)\subset V$. We will
show that for an appropriate $S$, $\var{f}{S}\ge m$. Hence
$\ind{f}{S}=\var{f}{S}+1\ge m+1$ and by Theorem~\ref{argupr} $f$ must have at
least $m+1$ fixed points in $V$.

Let us choose neighborhoods $U_i$ of exit continua
$K_i$ satisfying conditions listed below.

\begin{enumerate}

\item For $n_1<i\le n$ by assumption (2) of the theorem we may assume
that $f(U_i\cap X)\subset X$.

\item For $1\le i\le n_1$ we may assume that $d(U_i\cup Z_i,f(U_i))>0$.

\item We may assume that $T(X\cup \bigcup \ol{U_i})\subset V$ and
$\ol{U_i}\cap\ol{U_k}=\0$ for all $i\ne k$.

\item We may assume that every fixed point of $f$
contained in $\ol{U_i}$ is contained in $K_i$.

\end{enumerate}

Let $\{p_1,\dots,p_t\}$ be all fixed points of $f$ in $X\sm \bigcup_i K_i$ and
let $\{p_{t+1},\dots, p_m\}$ be all the fixed points contained in $\bigcup
K_i$. Observe that then by the choice of neighborhoods $U_i$ we have $p_i\in
X\sm \ol{\cup U_s}$ if $1\le i\le t$. Also, it follows that for each $j, t+1\le
j\le n$ there exists a unique $r_j, n_1<r_j\le n$ such that $p_j\in K_{r_j}$.
For each fixed point $p_j\in X$ choose a ray $R_j\subset \C\sm X$ landing on $p_j$, as specified in Definition~\ref{repout},
and  a small simple closed curve $S_j$
bounding a closed disk $D_j$ such that the following claims hold.

\begin{enumerate}
\item $D_i\cap R_j=\0$ for all $i\ne j$,

\item $f(\ol{S_j\sm X})\cap D_j=\0$.

\item $T(X\cup \bigcup_j D_j)\subset V$.

\item $[D_j\cup f(D_j)]\cap [D_k\cup f(D_k)]=\0$ for all $j\ne k$.

\item  $f(D_j\cap X)\subset X$.

\item Denote by $Q(j,s)$ the components of $S_j\sm X$; then there exists $Q(j,s(j))$,
 a component of $S_j\sm X$, with $\var{f}{Q(j,s(j))}>0$ and $Q(j,s(j))\cap R_j\ne\0$.

%If $p_j\in K_i$ (for some $t+1\le j\le n, n_1<i\le n$),

\item $[D_j\cup f(D_j)]\cap \bigcup \ol{U_i}=\0$ for all $1\le j\le t$.

\item If $t<j\le n$ then $[D_j\cup f(D_j)]\subset U_{r_j}$.

\end{enumerate}
Note that by (1) for all $i\ne j$, $\sh(Q(j,s(j))\cap Q(i,s(i))=\0$.
We need to choose a few constants. First choose $\e>0$ such that for all $x\in
X\sm \bigcup D_j$, $d(x,f(x))>3\e$. Then by continuity we can choose $\eta>0$
such that for each set $H\subset V$ of diameter less than $\eta$ we have
$\dia(H)+\dia(f(H))<\e$ and for each crosscut $C$ of $X$ disjoint from $\cup
D_j$ we have that $f(C)$ is disjoint from $C$ (observe that outside $\cup D_j$
all points of $X$ move by a bounded away from zero distance). Finally we choose
$\da>0$ so that the following inequalities hold:

\begin{enumerate}

\item $3\da<\e$,

\item $3\da<d(Z_i,Z_j)$ for all $i\ne j$,

\item \label{eU} $3\da<d(Z_i, [X\cup f(X)]\sm [Z_i\cup U_i])$,

\item $3\da<d(K_i,\C\sm U_i)$,

\item if $f(K_i)\cap Z_i=\0$, then $3\da<d(f(U_i),Z_i\cup U_i)$.

\end{enumerate}

Also, given a crosscut $C$ we can associate to its endpoints external angles
$\al, \be$ whose rays land at these endpoints from the appropriate side of $X$
determined by the location of $C$ (so that the open region of the plane
enclosed by a tight equipotential between $R_\al$ and $R_\be$, the segments of
the rays from the equipotential to the endpoints of $C$, and $C$ itself, is
disjoint from $X$). Thus we can talk about the angular measure of $Q(j, s(j))$;
denote by $\be$ the minimum of all such angular measures taken over all
crosscuts $Q(j, s(j))$.

Now, choose a bumping simple closed curve $S'$ of $X$ which satisfies two
conditions: all its links are (a) less than $\da$ in diameter, and (b) are of
angular measures less than $\be$. Clearly this is possible. Then we amend $S'$
as follows. Let us follow $S'$ in the positive direction starting at a link
outside $\cup D_j$. Then at some moment for the first time we will be walking
along a link of $S'$ which enters some $D_j$. As it happens, the link $L'$ of
$S'$ intersects some $Q(j, s)$ with endpoints $a, b$ and enters the shadow
$\sh(Q(j,s))$. Later on we will be walking outside $\sh(Q(j,s))$ moving along
some link $L''$. In this case we replace the entire segment of $S'$ from $L'$
to $L''$ by three links: the first one is a deformation of $L'$ which has the
same initial endpoint as $L'$ and the terminal point as $a$, then $Q(j, s)$,
and then a deformation of $L''$ which begins at $b$ and ends at the same
terminal point as $L''$. In this way we make sure that for all crosscuts $Q(j,
s)$ either they are links of $S'$ or they are contained in the shadow of a link
of $S'$. Moreover, by the choice of $\be$ crosscuts $Q(j, s(j))$ will have to
become links of our new bumping simple closed curve $S$. By the choice of
$\eta$ and by the properties of crosscuts $Q(j, s)$ it follows that any link of
$S$ is disjoint from its image, and for each $j$, $Q(j,s(j))\subset S$ and
$\var{f}{Q(j,s(j))}>0$.

We want to compute the variation of $S$. Each link $Q(j,s(j))$ contributes at
least $1$ towards $\var{f}{S}$, and we want to show that all other links have
non-negative variation. To do so we want to apply Lemma~\ref{posvar+}. Hence we
need to verify two conditions on a crosscut listed in Lemma~\ref{posvar+}. One
of them follows from the previous paragraph: all links of $S$ move off
themselves. However the other condition of Lemma~\ref{posvar+} may not be
satisfied by some links of $S$ because some of their endpoints may map off $X$.
To ensure that for our bumping simple closed curve endpoints $e$ of its links
map back into $X$ we have to enlarge links of $S$ and replace some of them by
their concatenations (this is similar to what was done in Theorem~\ref{fixpt}).
Then we will have to check if the new ``bigger'' links still have images
disjoint from themselves.

Suppose that $ X\cap S=A=\{a_0<a_1<\dots<a_n\}$ and $a_0\in A$ is such that
$f(a_0)\in X$ (the arguments similar to those in Theorem~\ref{fixpt} show that
we can may this assumption without loss of generality). Let $t'$ be minimal
such that $f(a_{t'})\not\in X$ and $t''>t'$ be  minimal such that
$f(a_{t''})\in X$. Then $f(a_{t'})\in Z_i$ for some $i$. Denote by $[a_l, a_r]$
a subarc of $S$ with the endpoints $a_r$ and $a_l$ and moving from $a_l$ to
$a_r$ is the positive direction. Since every component of $[a_{t'},a_{t''}]\sm
X$ has diameter less than $\da$, $f(a_t)\in Z_i\sm X$ for all $t'\le t<t''$.
Moreover, for $t'\le t<t''$, $a_t\not \in U_i$. To see this note that if
$f(K_i)\cap Z_i=\0$, then by the above made choices $f(U_i)\cap Z_i=\0$, and if
$f(K_i)\cap Z_i\ne\0$, then $f(U_i\cap X)\subset X$ by the assumption. Hence it
follows from the property (\ref{eU}) of the constant $\da$ that
$f([a_{t'-1},a_{t''}])\cap [a_{t'-1},a_{t''}]=\0$ and we can remove the points
$a_t$, for $t'\le t<t''$ from the partition $A$ of $S$. By continuing in the
same fashion we obtain a subset $A'\subset A$ such that for the closure of each
component $C$ of $S\sm A'$, $f(C)\cap C=\0$ and for both endpoints $a$ and $a'$
of $C$, $\{f(a),f(a')\}\subset X$. Moreover, for each $j$, $Q(j,j(s))$ is a
component of $S\sm A'$.

Now we can apply a variation of the standard argument sketched in Section 2
after Theorem~\ref{basic} and applied in the proof of Theorem~\ref{fixpt}; in
this variation instead of Theorem~\ref{basic} we use the fact that $f$
satisfies the argument principle. Indeed, by Theorem~\ref{FMOT} and
Lemma~\ref{posvar+}, $\ind{f}{S}\ge \sum_j \var{f}{Q(j,j(s))}+1\ge m+1$ and by
the Theorem~\ref{argupr} $f$ has at least $m+1$ fixed points in $T(S)\subset
V$, a contradiction.
\end{proof}

Theorem~\ref{locrot} implies the following

\begin{cor}\label{degenerate}
Suppose that $f$ and a non-degenerate $X$ satisfy all the conditions stated in
Theorem~\ref{locrot}. Then either $f$ does not repel outside $X$ at one of its
fixed points, or the local index at one of its fixed points is not equal to
$1$.
\end{cor}

The last lemma of this section gives a sufficient and verifiable condition for
a fixed point $a$ belonging to a locally invariant continuum $X$ to be such
that the map $f$ repels outside $X$; we apply the lemma in the next section.

\begin{lem}\label{repel}
Suppose that $f:\C\to\C$ is positively oriented, $X\subset\C$ is a continuum
and $p$ is a fixed point of $f$ such that:

\begin{enumerate}

\item
there exists a neighborhood $U$ of $p$ such that $f|_U$ is one-to-one and
$f(U\cap X)\subset X$,

\item
there exists a closed disk $D\subset U$ containing $p$ in its interior such
that $f(\partial D)\cap D=\0$ and $\partial D\sm X$ has at least two
components,

\item
there exists a  ray $R\subset\sphere\sm X$ from infinity such that
$\ol{R}=R\cup \{p\}$, $f|_R:R\to R$ is a homeomorphism and for each $x\in R$,
$f(x)$ separates $x$ from $\infty$ in $R$.

\end{enumerate}

Then there exists a component $C$ of $\partial D\sm X$ so that $C\cap R\ne\0$,  $\var{f}{C}=+1$
and $f$ repels  outside $X$ at $p$.
\end{lem}

\begin{proof}
We may assume that $X\sm U$ contains a continuum. Let $\disk^\iy=\sphere\sm
\ol{\disk}$ be the open disk at infinity and let $\vp:\disk^\iy\to \sphere\sm
X$ be a conformal map such that $\vp(\iy)=\iy$. Then $T=\vp^{-1}(R)$ is a ray
in $\disk^\infty$ which compactifies on a point $\hp\in S^1$. Let $Q_j$ be all
components of $\vp^{-1}(\partial D\sm X)$. Then each $Q_j$ is a crosscut of
$\disk^\iy$. Let $O=\{z\in\disk^\iy\mid f\circ \vp(z)\in\sphere\sm X$ and
define $F:O\to\disk^\iy$ by $F(z)=\vp^{-1}\circ f\circ \vp(z)$. Note that
$T\cup \bigcup Q_j\subset O$. We may assume that $\ol{Q_1}$ separates $\hp$
from $\iy$ in $\ol{\disk^\iy}$ and that no other $Q_j$ separates $Q_1$ from
$\iy$ in $\disk^\iy$.

\smallskip

\noindent \textbf{Claim.} $F(Q_1)$ separates $Q_1$ from $\iy$ in $\disk^\iy$.

\smallskip

\noindent \emph{Proof of Claim.} Let $T_\iy$ be the component of $T\sm Q_1$
which contains $\iy$ and let $T_p$ be the component of $\ol{T}\sm Q_1$ which
contains $\hp$. Let $a=\ol{T_\iy}\cap Q_1$ and $b=\ol{T_p}\cap Q_1$. Choose a
point $b'\in T_p$ very close to $b$ so that the subarc $[b,b']\subset T_p$ is
contained in $\vp^{-1}(D)$. Let $T'_p\subset T_p$ be the closed subarc from
$\hp$ to $b'$. Choose an open arc $A$ in the bounded component of $\disk^\iy\sm
Q_1$, very close to $Q_1$ from $a$ to the point $b'\in T_p$ so that
$f|_{T'_p\cap A\cup T_\iy}$ is one to-one. Put $Z=T'_p\cap A\cup T_\iy$, then
$Q_1\cap Z=\{a\}$ and $F(Q_1)\cap F(Z)=\{F(a)\}$. Since $F$ is a local
orientation preserving homeomorphism near $a$, $F(Z)$ enters the bounded
component of $\disk^\iy\sm F(Q_1)$ at $F(a)$ and never exits this component
after entering it. Moreover, if $q$ is an endpoint of $Q_1$, then points very
close to $q$ on $Q_1$ and their images are on the same side of $T$. Since
$F(a)$ separates $a$ from $\iy$ on $Z$ and an initial segment of $T_p$ (with
endpoint $\hp$) is contained in $F(Z)$, $\hp\in\sh(F(Q_1))$. This completes the
proof of the claim.

Let us compute the variation $\var{F}{Q_1}$ of the crosscut $Q_1$ with respect
to the continuum $S^1$. Since the computation is independent of the choice of
the Junction \cite{fmot07}, we can choose a junction $J_v$ with junction point
$v\in Q_1$ so that each of the three rays $R_+$, $R_i$ and $R_-$ intersect
$F(Q_1)$ in exactly one point. Hence $\var{F}{Q_1}=+1$. Since $\vp$ is an
orientation preserving homeomorphism, $\var{f}{\vp(Q_1)}=+1$ and we are done.

\end{proof}

\section{Applications}

The results in the previous section can be used to obtain results in complex
dynamics (see for example \cite{bco08}).
 We will show that in certain cases continua (e.g., impressions of
external rays) must be degenerate. Suppose that $P:\C\to\C$ is a complex
polynomial of degree $d$ with a connected  Julia set $J$. Let the filled-in
Julia set be denoted by $K=T(J)$. We denote the external rays of $K$ by
$R_\al$. It is well known \cite{douahubb85} that if the degree of $P$ is $d$
and $\sigma:\C\to\C$ is defined by $\sigma(z)=z^d$, then
$P(R_\al)=R_{\sigma(\al)}$.

Let for $\la\in\C$, $L_\la:\C\to\C$ be defined by $L_\la(z)=\la z$. Suppose
that $p$ is a fixed point in $J$ and $\la=f'(p)$ with $|\la|>1$ (i.e., $p$ is a
\emph{repelling} fixed point). Then there exists  neighborhoods $U\subset V$ of
$p$ and a conformal isomorphism $\vp:V\to\disk$ such that for all $z\in U$,
$P(z)=\vp^{-1}\circ L_\la \circ \vp(z)$. Now, a fixed point $p$ is
\emph{parabolic} if $P'(p)=e^{2\pi i r}$ for some rational number
$r\in\mathbb{Q}$. A nice description of the local dynamics at a parabolic fixed
point can be found in \cite{miln00}.

If $p$ is a repelling or parabolic fixed point then \cite{douahubb85} there
exist $k\ge 1$ external rays $R_{\al(i)}$ such that
$\si|_{\{\al(1),\dots,\al(k)\}}: \{\al(1),\dots,\al(k)\}\to
\{\al(1),\dots,\al(k)\}$ is a permutation, $P(R_{\al(i)})=R_{\si(\al(i))}$, for
each $j$, $R_{\al(j)}$ lands on $p$ and no other external rays land on $p$.
Also, if $P(R_{\al(i)})=R_{\al(i)}$ for some $i$, then $\si(\al(j))=\al(j)$ for
all $j$. It is known that two distinct external rays are not homotopic in the
complement of $K$. Given an external ray $R_\al$ of $K$ we denote by
$\Pi(\al)=\ol{R_\al}\sm R_\al$ the \emph{principle set of $\al$}, and by
$\imp(\al)$ the \emph{impression of $\al$} (see \cite{miln00}). Given a set $A\subset S^1$, we
extend the above notation by $\Pi(A)=\bigcup_{\al\in A} \Pi(\al)$ and
$\imp(A)=\bigcup_{\al\in A} \imp(\al)$. Let $X\subset K$ be a non-separating
continuum or a point such that:

\begin{enumerate}

\item \label{P1}
Pairwise disjoint non-separating continua/points $E_1\subset X, \dots, E_m\subset X$ and
finite sets of angles $A_1=\{\al^1_1, \dots, \al^1_{i_1}\}, \dots,
A_m=\{\al^m_1, \dots, \al^m_{i_m}\}$ are given with $i_k\ge 2, 1\le k\le m$.

\item \label{P2}
We have $\Pi(A_j)\subset E_j$ (so the set $E_j\cup
(\cup^{i_j}_{k=1} R_{\al^j_k})=E'_j$ is closed and connected).

\item \label{P3}
$X$ intersects a unique component $C$ of $\C\sm \cup E'_j$, such that $X\sm
\bigcup E_j=C\cap K$.

\end{enumerate}

We call such $X$ a \emph{general puzzle-piece} and call the continua $E_i$ the
\emph{exit continua} of $X$. Observe that if $U$ is a Fatou domain then either
a general puzzle-piece $X$ contains $U$, or it is disjoint from $U$. For each
$j$, the set $E'_j$ divides the plane into $i_j$ open sets which we will call
\emph{wedges (at $E_j$)}; denote by $W_j$ the wedge which contains $X\sm E_j$.

Let us now consider the condition (1) of Theorem~\ref{locrot}. It is easy to
see that applied ``as is'' to the polynomial $P$ at parabolic points it is
actually not true. Indeed, as explained above the local index at parabolic
fixed points at which the derivative equals $1$ is greater than $1$. And
indeed, in our case there are fixed rays landing at all fixed points, therefore
\cite{miln00} the derivatives at all the parabolic points in $X$ are equal to
$1$. The idea which allows us to solve this problem is that we can change our
map $P$ inside the parabolic domains in question without compromising the rest
of the arguments and making these parabolic points topologically repelling. The
thus constructed new map $g$ will satisfy conditions of Theorem~\ref{locrot}.

\begin{lem}\label{pararepel} Suppose that $X$ is a continuum and  $p\in X$ is a parabolic point
of a polynomial $f$ and $R$ is  a fixed external
which lands at $p$. Then $f$ repels outside $X$ at $p$.
\end{lem}
\begin{proof}Let $p\in X$ be a parabolic fixed point and let $F_i$ be the
parabolic domains containing $p$ in their boundaries $B_i$. Since there are
fixed rays landing at $p$, all $F_i$'s are forward invariant.
By a nice recent
result of Yin and Roesch \cite{roesyin08}, the boundary $B_i$ of each $F_i$ is
a simple closed curve and $f|_{B_i}$ is conjugate to the map $z\to z^{d(i)}$
for some $d(i)\ge 2$.  Let $\psi:F_i\to\disk$ be a conformal isomorphism. Since
$\bd(F_i)$ is a simple closed curve,  $\psi$ extends to a homeomorphism. Since
$f|_{B_i}$ is conjugate to the map $z\to z^{d(i)}$, it now follows that the map
$P|_{\ol{F_i}}$ can be replaced by a map topologically conjugate by $\psi$ to
the map $g_i(z)=z^{d(i)}$ on the closed unit disk. Let $g$ be the map defined
by $g(z)=P(z)$ for each $z\in\C\sm\bigcup F_i$ and $g(z)=g_i(z)$ when $z\in
F_i$. Then $g$ is clearly a positively oriented map.

The well-known analysis of the dynamics of $P$ around parabolic points
\cite{miln00} implies that $P$ repels points away from $p$ outside parabolic
domains $F_i$. In other words, we can find a sequence of simple closed curves
$S_i$ which satisfy conditions of Definition~\ref{toprepat} and show that $p$
is a topologically repelling point of $g$. Hence the local index $\ind{g}{p}$
at $p$ equals $1$. On the other hand, by Lemma~\ref{repel} and properties of
$X$ it follows that $g$ repels outside $X$ at $p$. Since $f$ and $g$ coincide
outside $X$, $f$ also repels at $p$.
\end{proof}
The following corollary follows from Theorem~\ref{locrot}.

\begin{cor}\label{pointdyn}
Suppose that $X\subset K$ is a non-separating continuum or a point. Then the
following claims hold for $X$.

\begin{enumerate}

\item\label{1}
Suppose that $X$ is a general puzzle-piece with exit continua $E_1, \dots, E_m$
such that either $P(E_i)\subset W_i$, or $E_i$ is a fixed point. If $X$ does
not contain an invariant parabolic domain, all fixed points which belong to $X$
are repelling or parabolic, and all rays landing at them are fixed, then $X$ is
a repelling or parabolic fixed point.

\item \label{2}
Suppose that $X\subset J$ is an invariant continuum, all fixed points which
belong to $X$ are repelling or parabolic, and all rays landing at them are
fixed. Then $X$ is a repelling or parabolic fixed point.

\end{enumerate}

\end{cor}

\begin{proof}
By way of contradiction we can assume that $X$ is not a point. Let us show that
no parabolic domain with a fixed point on its boundary can intersect $X$.
Indeed, in the case (2) $X\subset J$ and no Fatou domain intersects $X$, so there is nothing to
prove. In the case (1) observe that since $X$ is a general puzzle-piece, it has
to contain the closure of the entire parabolic domain with a fixed point, say,
$p$ on its boundary. Then the fact that all external rays landing at $p$ are
fixed implies that all parabolic domains containing $p$ in their boundaries are
invariant. Since by the assumptions $X$ contains no invariant parabolic domain,
it does not contain any of them. So, $X$ is disjoint from all parabolic domains
containing a fixed point in their boundaries.

To apply Theorem~\ref{locrot} we need to verify that its conditions apply. It
is easier to check the condition (2) first. To do so, observe first that
$f(X)\cap X\ne \0$. Indeed, otherwise no set $E_i$ is a fixed point and $f(X)$
must be contained in one of the wedges formed by some $E'_l$ but not in the
wedge $W_l$. This implies that $E_l$ neither is a fixed point, nor is mapped in
$W_l$, a contradiction. Thus, $f(X)\cap X\ne \0$ and we can think of $f(X)$ as
a continuum which ``grows'' out of $X$. Now, any component of $f(X)\sm X$ which
intersects $E_k$ for some $k$ must be contained in one of the wedges at $E_k$,
but not in $W_k$. Take the closure of their union and then its topological hull
union $E_i$ and denote it by $Z_i$. It is easy to check now that with these
sets $Z_i$ the map $P$ scrambles the boundary of $X$. Moreover, if $E_i$ is
mapped into $W_i$ then clearly $P(E_i)\cap Z_i=\0$ (because $Z_i$ is contained
in the other wedges at $E_i$ but is disjoint from $W_i$). On the other hand, if
$E_i$ is a fixed point then it is a repelling or parabolic fixed point with a
few external fixed rays landing at it. Hence in a small neighborhood $U_i$ of
$E_i$ the intersection $U_i\cap X$ maps into $X$ as desired in the condition
(2) of Theorem ~\ref{locrot}.

By Lemmas~\ref{repel} and \ref{pararepel} $P$ repels outside $X$ at any fixed
point in $X$. Moreover, using the map $g$ constructed in the proof of Lemma~\ref{pararepel},
we see that $g$ is topologically repelling at $p$ and, hence
$\ind{g}{p}=+1$.  Hence the conditions  of Theorem~\ref{locrot} are  satisfied for the map $g$.
Thus, by Theorem~\ref{locrot} we conclude that $X$ is a point as desired.
\end{proof}

The following is an immediate corollary of Theorem~\ref{pointdyn}.

\begin{cor}\label{rot-neutr}
Suppose that for a non-separating \emph{non-degenerate} continuum $X\subset K$
one of the following facts hold.

\begin{enumerate}

\item
$X$ is a general puzzle-piece with exit continua $E_1, \dots, E_m$ such that
either $P(E_i)\subset W_i$, or $E_i$ is a fixed point.

\item
$X\subset J$ is an invariant continuum.

\end{enumerate}

Then either $X$ contains a non-repelling and non-parabolic fixed point, or $X$
contains an invariant parabolic domain, or $X$ contains a repelling or
parabolic fixed point at which a non-fixed ray lands.
\end{cor}

Finally, the following corollary is useful in proving the degeneracy of certain
impressions and establishing local continuity of the Julia set at some points.

\begin{cor}
Let $P:\C\to\C$ be a complex polynomial and  $R_\al$ is a fixed external ray
landing on a repelling or parabolic fixed point $p\in J$. Suppose that
$T(\imp(\al))$ contains only repelling or parabolic periodic points. Then
$\imp(\al)$ is degenerate.
\end{cor}

\begin{proof}
Let $X=\imp(\al)$. Since $R_\al$ is a fixed external ray, $P(X)\subset X$.
Clearly $P$ does not rotate at $p$. Suppose that $p'$ is another fixed point of
$P$ in $X$ and $R_\ba$ is an external ray landing at $p'$. Then $P(R_\ba)$ also
lands on $p'$. If $P$ rotates at $p'$, then $p'$ is a cut point of $X$. This
would contradict the fact that $X=\imp(\al)$. Hence $P$ does not rotate at any
fixed point in $X$ and the result follows from Corollary~\ref{pointdyn}.
\end{proof}

\bibliographystyle{amsalpha}
\bibliography{/lex/references/refshort}

\providecommand{\bysame}{\leavevmode\hbox to3em{\hrulefill}\thinspace}
\providecommand{\MR}{\relax\ifhmode\unskip\space\fi MR }
\providecommand{\MRhref}[2]{%
  \href{http://www.ams.org/mathscinet-getitem?mr=#1}{#2}
}
\providecommand{\href}[2]{#2}

\end{document}